\documentclass[12pt]{amsproc}
\newtheorem{theorem}{\bf Theorem}[section]
\newtheorem{lemma}[theorem]{\bf Lemma}

\newtheorem{corollary}[theorem]{\bf Corollary}

\author{Pavel Shumyatsky}
\address{Department of Mathematics, University of Brasilia, 70910 Bras\'ilia DF, Brazil}
\email{pavel@unb.br}

\author{Danilo  Silveira}
\address{Department of Mathematics, Federal University of Goi\'as, 75704-020 Catal\~ao GO, Brazil}
\email{sancaodanilo@gmail.com}


\thanks{This work was supported by the Conselho Nacional de Desenvolvimento Cient\'{\i}fico e Tecnol\'ogico (CNPq), Brazil. }

\title[On finite groups in which commutators are Engel]{On finite groups in which commutators are covered by Engel subgroups}

\begin{document}

\begin{abstract}  
Let $m,n$ be positive integers and $w$ a multilinear commutator word. Assume that $G$ is a finite group having subgroups $G_1,\ldots,G_m$ whose union contains all $w$-values in $G$. Assume further that all elements of the subgroups $G_1,\ldots,G_m$ are $n$-Engel in $G$. It is shown that the verbal subgroup $w(G)$ is $s$-Engel for some $\{m,n,w\}$-bounded number $s$.
\end{abstract}

\maketitle

\section{Introduction}

Given a group-word $w=w(x_1,\dots,x_k)$, we think of it primarily as a function of $k$ variables defined on any group $G$. We denote by $w(G)$ the verbal subgroup of $G$  corresponding to the word $w$, that is, the subgroup generated by the $w$-values in $G$. When the set of all $w$-values in $G$ is contained in a union of subgroups we wish to know whether the properties of the covering subgroups have impact on the structure of the verbal subgroup $w(G)$. The reader can consult the articles \cite{as,surveyrendiconti, DMS1,DMS-revised,PDM,Snilp} for results on countable coverings of $w$-values in profinite groups. 

The purpose of this paper is to prove the following result.

\begin{theorem}\label{main} Let $m,n$ be positive integers and $w$ a multilinear commutator word. Assume that $G$ is a finite group having subgroups $G_1,\ldots,G_m$ whose union contains all $w$-values in $G$. Assume further that all elements of the subgroups $G_1,\ldots,G_m$ are $n$-Engel in $G$. Then $w(G)$ is $s$-Engel for some $\{m,n,w\}$-bounded number $s$.
\end{theorem}
 Here and throughout the article we use the expression ``$\{a,b,\dots\}$-bounded'' to abbreviate ``bounded from above in terms of  $a,b,\dots$ only''.

Recall that multilinear commutators are words which are obtained by nesting commutators, but using always different variables. More formally, the word $w(x) = x$ in one variable is a multilinear commutator; if $u$ and $v$ are  multilinear commutators involving different variables then the word $w=[u,v]$ is a multilinear commutator, and all multilinear commutators are obtained in this way. The number of variables involved in a multilinear commutator $w$ is called the weight of $w$.

Also, recall that a group $G$ is called an Engel group if for every $x,y\in G$ the equation $[y,x,x,\dots,x]=1$ holds, where $x$ is repeated in the commutator sufficiently many times depending on $x$ and $y$. The long commutators $[y,x,\dots,x]$, where $x$ occurs $i$ times, are denoted by $[y,{}_i\,x]$. An element $x\in G$ is (left) $n$-Engel if $[y,{}_n\,x]=1$ for all $y\in G$. A group $G$ is $n$-Engel if $[y,{}_n\,x]=1$ for all $x,y\in G$. Currently, finite $n$-Engel groups are understood fairly well. A theorem of Baer says that finite Engel groups are nilpotent (see \cite[Theorem 12.3.7]{Rob}). More specific properties of finite $n$-Engel groups can be found for example in a theorem of Burns and Medvedev quoted as Theorem \ref{BM} in Section 3 of this paper. The interested reader is refered to the survey \cite{trau} and references therein for further results on finite and residually finite 
Engel groups.

In the next section we describe the Lie-theoretic machinery that will be used in the proof of Theorem \ref{main}. The proof of the theorem is given in Section 3.

\section{Associating a Lie ring to a group}

There are several well-known ways to associate a Lie ring to a group $G$ (see \cite{Huppert2,Kh,S2000}). For the reader's convenience we will briefly describe the construction that we are using in the present paper.

A series of subgroups $$G=G_1\geq G_2\geq\cdots\eqno{(*)}$$ is called an $N$-series if it satisfies $[G_i,G_j]\leq G_{i+j}$ for all $i,j$. Obviously any $N$-series is central, i.e. $G_i/G_{i+1}\leq Z(G/G_{i+1})$ for any $i$. Given an $N$-series $(*)$, let $L^*(G)$ be the direct sum of the abelian groups $L_i^*=G_i/G_{i+1}$, written additively. Commutation in $G$ induces a binary operation $[,]$ in $L^*(G)$. For homogeneous elements $xG_{i+1}\in L_i^*,yG_{j+1}\in L_j^*$ the operation is defined by $$[xG_{i+1},yG_{j+1}]=[x,y]G_{i+j+1}\in L_{i+j}^*$$ and extended to arbitrary elements of $L^*(G)$ by linearity. It is easy to check that the operation is well-defined and that $L^*(G)$ with the operations $+$ and $[,]$ is a Lie ring. 

In this paper we use the above construction in the cases where $(*)$ is either the lower central series of $G$ or the $p$-dimension central series, also known under the name of Zassenhaus-Jennings-Lazard series (see \cite[p.\ 250]{Huppert2} for details). In the former case we denote the associated Lie ring by $L(G)$. In the latter case $L^*(G)$ can be viewed as a Lie algebra over the field with $p$ elements. We write $L_p(G)$ for the subalgebra generated by the first homogeneous component $G_1/G_2$. Usually nilpotency of $L^*(G)$ has strong effect on the structure of $G$. In particular, $L(G)$ is nilpotent of class $c$ if and only if the group $G$ is nilpotent of class $c$. Nilpotency of $L_p(G)$ also leads to strong conclusions about $G$. The proof of the following theorem can be found in \cite{KS}.

\begin{theorem}\label{finiteLazard}
Let $P$ be a $d$-generated finite $p$-group and suppose that $L_p(G)$ is nilpotent of class $c$. Then $P$  has a powerful characteristic subgroup of  $\{p,c,d\}$-bounded index.
\end{theorem}

Recall that powerful $p$-groups were introduced by Lubotzky and Mann in \cite{luma}. They have many nice properties, some of which are listed in the next section.

Thus, criteria of nilpotency of Lie algebras provide effective tools for applications in group theory.

Let $X$ be a subset of a Lie algebra $L$. By a commutator in elements of $X$ we mean any element of $L$ that can be obtained as a Lie product of elements of $X$ with some system of brackets. If $x,y$ are elements of $L$, we define inductively 
$$[x,_0y]=x \text{ and } [x,_iy]=[[x,_{i-1}y],y]\text{ for all positive integers } i.$$  
As usual, we say that an element $a\in L$ is ad-nilpotent if there exists a positive integer $n$ such that $[x,_na]=0$ for all $x\in L$. If $n$ is the least integer with the above property, then we say that $a$ is ad-nilpotent of index $n$. 

The next theorem is a deep result of Zelmanov with many applications to group theory. It was announced by Zelmanov in \cite{Z1,Z0}. A detailed proof was  published in \cite{zenew}.

\begin{theorem}\label{Z1992}	Let $L$ be a Lie algebra over a field and suppose that $L$ satisfies a polynomial identity. If $L$ can be generated by a finite set $X$ such that every commutator in elements of $X$ is ad-nilpotent, then $L$ is nilpotent.
\end{theorem}

Theorem \ref{Z1992} admits  the following  quantitative version (see for instance \cite{KS}).

\begin{theorem}\label{Z1} Let $L$ be a Lie algebra over a field $K$. Assume that $L$ is generated by $m$ elements such that each commutator in the generators is ad-nilpotent of index at most $n$. Suppose that $L$ satisfies a polynomial identity $f\equiv 0$. Then $L$ is nilpotent of $\{f,K,m,n\}$-bounded class.
\end{theorem}

As usual, $\gamma_i(L)$ denotes the $i$th term of the lower central series of $L$. The following Lie-ring variation on the theme of Theorem \ref{Z1992} is a particular case of \cite[Proposition 2.6]{shusa}.
 
\begin{theorem}\label{BazaRing} Let $L$ be a Lie ring satisfying a polynomial identity $f\equiv 0$. Assume that $L$ is generated by $m$ elements such that every commutator in the 	generators is ad-nilpotent of index at most $n$. Then there exist positive integers $e$ and $c$ depending only on $f,m$ and $n$ such that  $e\gamma_c(L)=0$.
\end{theorem}

\section{Proof of the main theorem}

It will be convenient first to prove Theorem \ref{main} in the particular case where $w=\delta_k$ is a derived word. Recall that the derived words $\delta_k$, on $2^k$ variables, are defined recursively by 
\begin{center}

$\delta_0=x_1 \text { and } \delta_k=[\delta_{k-1}(x_1,\ldots,x_{2^{k-1}}),\delta_{k-1}(x_{2^{k-1}+1},\ldots,x_{2^{k}})]$ for $k\geq 1$.
\end{center}

The verbal subgroup corresponding to the word $\delta_k$ in a group $G$ is the familiar $k$th term of the derived series of $G$ denoted by $G^{(k)}$.

\begin{lemma}\label{delta}
Let $m,n,k$ be positive integers, and let $G$ be a finite group with subgroups $G_1,\ldots,G_m$ whose union contains all $\delta_k$-values in $G$. If all elements of the subgroups $G_1,\ldots,G_m$ are $n$-Engel in $G$, then $G^{(k)}$ is $s$-Engel for some $\{k,m,n\}$-bounded number $s$.
\end{lemma}

A subset $X$ of a group $G$ is called commutator-closed if $[x,y]\in X$ whenever $x,y\in X$. The fact that in any group the set of all $\delta_k$-values is commutator-closed will be used without explicit references.

The proof of Lemma \ref{delta} will require the following two lemmas which  were obtained in \cite[Lemma 3.1]{as} and \cite[4.1]{shusa}, respectively.

\begin{lemma}\label{lemmaSA}
Let $G$ be a nilpotent group generated by a commutator-closed subset $X$ which is contained in a union of finitely many  subgroups $G_1,G_2,\ldots,G_m$. Then $G=G_1G_2\cdots G_m$.
\end{lemma}

\begin{lemma}\label{gru} 
Let $G$ be a group generated by $m$  elements which are $n$-Engel. 
If $G$ is soluble with derived length $d$, then $G$ is nilpotent of
 $\{d, m, n\}$-bounded class.
\end{lemma}

The proof of Lemma \ref{delta} requires the concept of powerful $p$-groups. A finite $p$-group $P$ is said to be powerful if and only if $[P,P]\leq P^p$ for $p\neq 2$ (or $[P,P]\leq P^{4}$ for $p=2$), where $P^{i}$  denotes   the subgroup of $P$ generated by all $i$th powers. If $P$ is a powerful $p$-group, then the subgroups $\gamma_{i}(P), P^{(i)}$ and $P^{i}$ are also powerful. Moreover, for given positive integers $n_1,\ldots,n_j$, it follows, by repeated applications of \cite[Propositions 1.6 and 4.1.6]{luma}, that $$[P^{n_1},\ldots,P^{n_j}]\leq \gamma_j(P)^{n_{1}\cdots n_j}.$$ Furthermore if a powerful $p$-group $P$ is generated by $d$ elements, then any subgroup of $P$ can be generated by at most $d$ elements and $P$ is a product of $d$ cyclic subgroups. For more details we refer the reader to \cite[Chapter 11]{Kh}.

\begin{proof}[{\bf Proof of Lemma \ref{delta}}]
By the hypothesis, each $\delta_k$-value is $n$-Engel in $G$. Hence, Baer's theorem \cite[Theorem 12.3.7]{Rob} implies that $G^{(k)}$ is nilpotent. Replacing if necessary $G_i$ by $G_i\cap G^{(k)}$, we can assume that all subgroups $G_i$ are contained in $G^{(k)}$. Then, by Lemma \ref{lemmaSA},  $G^{(k)}=G_1G_2\cdots G_m$.  

Choose arbitrary elements $a,b\in G^{(k)}$. It is sufficient to show that the subgroup $\langle a,b\rangle$ is nilpotent of $\{m,n,k\}$-bounded class. Write $a=a_1\cdots a_m$ and $b=b_1\cdots b_m$, where $a_i$ and $b_i$ belong to $G_i$. Let $H$ be the subgroup generated by the elements $a_i, b_i$ for $i=1,\ldots, m.$ Since the subgroup $\langle a,b\rangle$ is contained in $H$, it is enough to show that $H$ is nilpotent of $\{m,n,k\}$-bounded class. Observe that the generators $a_i,b_i$ of $H$ are $n$-Engel elements. Thus, in view of Lemma \ref{gru}, it is sufficient to prove that $H$ has $\{m,n,k\}$-bounded derived length. Since $H$ is nilpotent, we need to show that each Sylow $p$-subgroup $P$ of $H$ has $\{m,n,k\}$-bounded derived length.

Obviously, $P$ can be generated by 2$m$ elements each of which is $n$-Engel.  Set $R=P^{(k)}$. We will now prove that $R$ can be generated by $\{m,n,k\}$-boundedly many, say $r$, elements. Note that by Burnside Basis Theorem \cite[Theorem 5.3.2]{Rob}, it is sufficient to show that the Frattini quotient $R/\Phi(R)$ has  $\{m,n,k\}$-boundedly  many generators. The quotient $P/\Phi(R)$ has derived length at most $k+1$. Thus, Lemma \ref{gru} implies that  $P/\Phi(R)$  has $\{m,n,k\}$-bounded nilpotency class. It follows that $R/\Phi(R)$ can be generated with $\{m,n,k\}$-boundedly  many elements. This is also true for $R$.

Next, we will show that $R$ has $\{m,n,k\}$-bounded derived length. By the Burnside Basis Theorem, $R$ is generated by $r$ $\delta_k$-values which are $n$-Engel elements. Let $L_1=L(R)$ be the Lie ring associated to $R$ using the lower central series. The proof of \cite[Theorem 1]{WZ} shows that since $R$ satisfies the identity $[y,_n\delta_k(x_1,\ldots,x_{2^k})]\equiv1$, the Lie ring $L_1$ satisfies the linearized version of the identity $[y,_n\delta_k(x_1,\ldots,x_{2^k})]\equiv0$. Further, each commutator in the generators of $L_1$ corresponding to $\delta_k$-values in $R$ is ad-nilpotent of index at most $n$. By Theorem \ref{BazaRing}, there exist positive integers $e$ and $c$, depending only on $k,m$ and $n$, such that $e\gamma_c(L_1) = 0$. If $p$ is not a divisor of $e$, we have $\gamma_c(L_1) = 0$ and so the group $R$ is nilpotent of class at most $c-1$. In what follows we assume that $p$ is a divisor of $e$. Note that in this case $p$ is bounded in terms of $k,m$ and $n$.

Let $L_2=L_p(R)$ be the Lie algebra associated to $R$ using the $p$-dimensional series. Applying Theorem \ref{Z1} we deduce that $L_2$ is nilpotent with $\{m,n,k\}$-bounded nilpotency class. Hence, by Theorem \ref{finiteLazard}, $R$ has a powerful subgroup $N$ of $\{m,n,k\}$-bounded index. It is now sufficient to show that $N$ has $\{m,n,k\}$-bounded derived length. 

Since the index of $N$ in $R$ is $\{m,n,k\}$-bounded, it follows that $N$ can be generated with $\{m,n,k\}$-boundedly many elements, say $t$. Taking into account that $N$ is powerful, we deduce that all subgroups of $N$ can be generated by at most $t$ elements, and the $k$th derived subgroup $N^{(k)}$ is also powerful. We now look at the Lie ring $L(N^{(k)})$ associated to $N^{(k)}$.

By Theorem \ref{BazaRing}, there exist positive integers $e_1,c_1$ depending only on $k,m$ and $n$, such that $e_1\gamma_{c_1}(L(N^{(k)})) = 0$. Since $P$ is a $p$-group, we can assume that $e_1$ is a $p$-power. Set $R_1=(N^{(k)})^{e_1^{2^k}}=(N^{e_1})^{(k)}$.  

Note that if $p\neq 2$, then 
$$[R_1,R_1]\leq [N^{(k)},N^{(k)}]^{e_1^{2^k}e_1^{2^k}}\leq (N^{(k)})^{pe_1^{2^k}e_1^{2^k}}=R_1^{pe_1^{2^k}}.$$
If $p=2$, then we have $$[R_1,R_1]\leq R_1^{4e_1^{2^k}}.$$

Since $e_1\gamma_{c_1}(L(R_1)) = 0$, we deduce that $\gamma_{c_1}(R_1)^{e_1}\leq \gamma_{c_1+1}(R_1)$. Taking into account that $R_1$ is powerful, if $p\neq2$ we obtain that $$\gamma_{c_1}(R_1)^{e_1}\leq \gamma_{c_1+1}(R_1)=[R'_1,_{c_1-1}R_1]\leq [R_1^{pe_1^{2^k}},_{c_1-1}R_1]\leq \gamma_{c_1}(R_1)^{pe_1^{2^k}}$$
If $p=2$, we obtain that
$$\gamma_{c_1}(R_1)^{e_1}\leq\gamma_{c_1}(R_1)^{4e_1^{2^k}}.$$
Hence, $\gamma_{c_1}(R_1)^{e_1}=1$. Since $\gamma_{c_1}(R_1)$ is powerful and generated by at most $t$ elements, we conclude that $\gamma_{c_1}(R_1)$ is a product of at most $t$ cyclic subgroups.  Hence the order of $\gamma_{c_1}(R_1)$ is at most $e_1^t$. It follows that the derived length of $R_1$ is $\{k,m,n\}$-bounded. Recall that $N^{(k)}$ is a powerful $p$-group and $R_1 =(N^{(k)})^{e_1^{2^k}}$. It follows that the derived length of $N^{(k)}$ is $\{k,m,n\}$-bounded. Hence, the derived length of $P$ is $\{k,m,n\}$-bounded, as required. The proof is now complete.
\end{proof}

The next lemma is well-known (see for example \cite[Lemma 4.1]{S-outer} for a proof).
\begin{lemma}\label{delta-outer} Let $G$ be a group and $w$ a multilinear commutator word of weight $k$. Then every $\delta_k$-value in $G$ is a $w$-value.
\end{lemma}

The proof of Theorem \ref{main} will require the following result, due to Burns and Medvedev \cite{BM}.

\begin{theorem}\label{BM}
Let $n$ be a positive integer. There exist constants $c$ and $e$  depending only on $n$ such that if $G$ is a finite $n$-Engel group, then the exponent of $\gamma_c(G)$ divides $e$.
\end{theorem}

Another useful result which we will need is the next theorem \cite[Theorem B]{FAM}.

\begin{theorem}\label{FAM}
Let $w$ be a multilinear commutator word, and let $G$ be a soluble group. Then there exists a series of subgroups from 1 to $w(G)$ such that:
\begin{itemize}
\item all subgroups of the series are normal in $G$;
\item every section of the series is abelian and can be generated by $w$-values all of whose powers are also $w$-values.
\end{itemize}
Furthermore, the length of this series only depends on the word $w$ and on the derived length of $G$.
\end{theorem}
\begin{corollary}\label{cccc}
Assume the hypotheses of Theorem \ref{main} and suppose additionally that $G$ is soluble with derived length $k$. Then each element of $w(G)$ can be written as a product of $\{k,m\}$-boundedly many elements from the subgroups $G_1,\dots,G_m$.
\end{corollary}
\begin{proof} Let $1=A_0\leq A_1\leq\dots\leq A_u=w(G)$ be a series as in Theorem \ref{FAM}. Arguing by induction on $u$ it is sufficient to show that each element of $A_1$ can be written as a product of $\{k,m\}$-boundedly many elements from the subgroups $G_1,\dots,G_m$. Since $A_1$ is abelian and generated by $w$-values each of which lies in some $G_i$, we deduce that $A_1$ is the product of subgroups of the form $A_1\cap G_i$. The result follows. 
\end{proof}
 Now we are ready to prove Theorem \ref{main}.
\begin{proof}[{\bf Proof of Theorem \ref{main}}.] Recall that $w$ is a multilinear commutator word. Since each $w$-value in $G$ is $n$-Engel, Baer's theorem implies that the verbal subgroup $w(G)$ is nilpotent. Let $d$ be the weight of the word $w$. Combining Lemmas \ref{delta-outer} and \ref{delta} we deduce that $G^{(d)}$ is $u$-Engel for some $\{m,n,d\}$-bounded number $u$. Theorem \ref{BM} shows that there exists an $\{m,n,d\}$-bounded number $c$ such that $\gamma_c(G^{(d)})$ has $\{m,n,d\}$-bounded exponent. It follows that there is an $\{m,n,d\}$-bounded number $k$ such that $M=G^{(k)}$ has $\{m,n,d\}$-bounded exponent. 

Choose arbitrary elements $a,b\in w(G)$. We will show that the subgroup $\langle a,b\rangle$ is nilpotent of $\{m,n,w\}$-bounded class. Corollary \ref{cccc} shows that any element in $w(G)/M$ can be written as a product of  $\{m,n,w\}$-boundedly many, say $r$, $w$-values. Thus, we can write $a=a_1\cdots a_rm_1$ and $b=b_1\cdots b_rm_2$, where  $a_i,b_i$  are $w$-values for $i=1,\ldots,r$ and $m_1,m_2$ belong to $M$. Let $H$ be the subgroup generated by all these elements, that is, $$H=\langle a_1,\ldots,a_r,b_1,\ldots,b_r,m_1,m_2\rangle.$$ Note that the subgroup $\langle a,b\rangle$ is contained in $H$, and therefore it is sufficient to show that $H$ is nilpotent of $\{m,n,w\}$-bounded class.

Set $N=M\cap H$ and let $\bar{H}$ be the quotient group $H/\Phi(N)$. Note that the image of $N$ in $\bar{H}$ is an abelian group, and so the images of $m_1,m_2$ in $\bar{H}$ are 2-Engel. Note also that the derived length of $\bar{H}$ is at most $k+1$. Lemma \ref{gru} yields that the nilpotency class of $\bar{H}$ is $\{m,n,w\}$-bounded. Thus, we get that the image of $N$ in $\bar{H}$ has $\{m,n,w\}$-boundedly many generators. Of course, this is true also for $N$. Recall that the exponent of $N$ is $\{m,n,w\}$-bounded, and so we obtain from the positive solution of the restricted Burnside problem \cite{Z1,Z0} that the order of $N$ is $\{m,n,w\}$-bounded. Since $H$ is nilpotent, there is an $\{m,n,w\}$-bounded number $t$ such that $N$ is contained in the $t$th term $Z_t(H)$ of the upper central series of $H$.
Consequently $H$ is nilpotent of $\{m,n,w\}$-bounded class, as required. The proof is now complete.
\end{proof}

\end{document}